\documentclass[reqno]{amsart}
\usepackage{amsfonts}
\usepackage{amssymb}
\usepackage{amsmath}

\newtheorem{theorem}{Theorem}[section]

\newtheorem{corollary}[theorem]{Corollary}

\newtheorem{lemma}[theorem]{Lemma}

\newtheorem{proposition}[theorem]{Proposition}

\theoremstyle{remark}
\newtheorem{remark}[theorem]{Remark}

\numberwithin{equation}{section}

\newcommand{\ep}{\varepsilon}

\renewcommand{\div}{{\rm div}}

\begin{document}
\title[Wave equations with degenerate memory]{Regular global attractors for wave equations with degenerate memory}
\author[J. L. Shomberg]{Joseph L. Shomberg}
\address{Joseph L. Shomberg, Department of Mathematics and Computer Science,
Providence College, Providence, RI 02918, USA}
\email{jshomber@providence.edu}
\date{\today}

\begin{abstract}
We consider the wave equation with degenerate viscoelastic dissipation recently examined in Cavalcanti, Fatori, and Ma, \emph{Attractors for wave equations with degenerate memory}, J. Differential Equations (2016).
Under certain extra assumptions (namely on the nonlinear term), we show the existence of a compact attracting set which provides further regularity for the global attractor and show that it consists of regular solutions.
\end{abstract}

\maketitle
\tableofcontents

\section{Introduction}

An elastic body perturbed from equilibrium may undergo a restoring force subject to both frictional and viscoelastic dissipation mechanisms.
The problem under consideration is the wave equation with degenerate viscoelastic dissipation in the unknown $u=u(x,t)$
\begin{align}
u_{tt} - \Delta u + \int_0^\infty g(s)\div[a(x)\nabla u(t-s)]ds + b(x)u_t + f(u) = h(x) \quad \text{in} \ \Omega\times\mathbb{R}^+,  \label{tpde}
\end{align}
defined on a bounded domain $\Omega$ in $\mathbb{R}^3$ with smooth (at least class $\mathcal{C}^2$) boundary $\Gamma$.
The equation is subject to Dirichlet boundary conditions
\begin{align}
u(x,t) = 0 \quad \text{on} \ \Gamma\times\mathbb{R}^+,  \label{tbc}
\end{align}
and the initial conditions
\begin{align}
u(x,0) = u_0(x) \quad \text{and} \quad u_t(x,0) = u_1(x) \quad \text{at} \ \Omega\times\{0\}.  \label{tic}
\end{align}

This problem was recently treated, to the extent of global well-posedness and global attractors, in \cite{CFM-16}.
The novelty here being the degenerate nature of the viscoelasticity. 
Similar problems have yielded several important results as well.
We mention some other works concerning semilinear wave equations with memory.
On the asymptotic behavior of solutions (in the sense of global attractors) see \cite{Conti-Pata-2005,CPS05,FePeAn2016,Pata-Zucchi-2001,Plinio&Pata09,PPZ08}, and on rates of decay of solutions one can also see \cite{LiZh2011,santos07,Tahamtani_799}.

To the problem under consideration here, the well-posedness was carried out under the guise of semigroup methods.
Here, local mild solutions and regular (or ``strong'' solutions) are obtained using the fact that the underlying operator is the infinitesimal generator of a strongly continuous semigroup of contractions on the Hilbertian phase space $\mathcal{H}$, and the other condition naturally being that the nonlinear term defines a locally Lipschitz continuous functional also on $\mathcal{H}$. 

The main result concerning the asymptotic behavior of \eqref{tpde}-\eqref{tic} in \cite{CFM-16} consists in demonstrating the existence of a finite dimensional global attractor for the semidynamical system $(\mathcal{H},S(t)).$
For this, the authors of \cite{CFM-16} rely on \cite[Proposition 7.9.4 and Theorem 7.9.6]{Ch-La-10}.
That is, the problem is of the asymptotically smooth gradient system class where the set of stationary points is bounded.
The so-called quasi-stability of the dynamical system $(\mathcal{H},S(t))$ involves finding a suitable (relatively) compact seminorm on $\mathcal{H}$ (i.e., the approach is similar to finding a global attractor via an $\alpha$-contraction method).
Instead of characterizing the global attractor as the omega-limit set of some bounded absorbing set $\mathcal{B}$ in $\mathcal{H},$ i.e. $\mathcal{A}=\omega(\mathcal{B}),$ the global attractor in this work is characterized with properties from the gradient system so that the global attractor is described by the union of unstable manifolds connecting the set of stationary points $\mathcal{N}$, i.e. $\mathcal{A}=\mathbb{M}^u(\mathcal{N})$.
Unlike the methods used to prove the existence of a global attractor by virtue of the former characterization, in the latter no (explicit) bounded absorbing set $\mathcal{B}$ nor any (explicit) uniform bound on solutions is used to prove the existence of the global attractor. 
Finally, it seems that an explicit bound in terms of some of the parameters of the problem (Lipschitz constant, etc.) can be given to the fractal dimension of the global attractor (indeed, see \cite[Theorem 3.4.5]{Chueshov-15}).
These results are obtained without assuming the two damping terms satisfy a geometric control condition (cf. e.g. \cite{Joly-Camille-13}). 

To treat the memory term, we define a past history variable using the {\em{relative displacement history}}, for all $x\in\Omega\subset \mathbb{R}^3$ and $s,t\in\mathbb{R}^+$,
\begin{equation}
\eta^t(x,s) := u(x,t) - u(x,t-s).  \label{memory-defn}
\end{equation}
In order for this formulation to make sense, we also need to prescribe the past history of $u(x,t)$, $t<0$.
Observe, from \eqref{memory-defn} we readily find the useful identity
\[
\int_0^\infty g(s)\div[a(x)\nabla u(t-s)]ds = - \int_0^\infty g(s)\div[a(x)\nabla\eta^t(s)]ds + k_0\div[a(x)\nabla u(s)],
\]
where $k_0:=\int_0^\infty g(s) ds$ assumed to be sufficiently small below (see \eqref{small-g}).
Thus, equations \eqref{tpde}-\eqref{tic} have an equivalent form in the unknowns $u=u(x,t)$ and $\eta^t=\eta^t(x,s),$ for all $x\in\Omega$ and $s,t\in\mathbb{R}^+$,
\begin{align}
u_{tt} & - \div[(1-k_0a(x))\nabla u]-\int_0^\infty g(s)\div[a(x)\nabla\eta^t(s)]ds + b(x)u_t + f(u) = h(x),  \label{pde-1} \\
\eta_t & = -\eta_s + u_t,  \label{pde-2}
\end{align}
with boundary conditions, for all $(x,t)\in\Gamma\times\mathbb{R}^+$,
\begin{align}
u(x,t) = 0 \quad \text{and} \quad \eta^t(x,s) = 0,  \label{bc}
\end{align}
and the following initial conditions at $t=0,$
\begin{align}
u(x,0) = u_0(x), \quad u_t(x,0) = u_1(x) \quad \text{and} \quad \eta^t(x,0) = 0, \quad \eta^0(x,s) = \eta_0(x,s).  \label{ic}
\end{align}

In this article, we aim to provide a regularity result to the global attractors found  in \cite{CFM-16} for the problem \eqref{tpde}-\eqref{tic}.

\section{Preliminaries}

This section contains a summary of the assumptions and main results of \cite{CFM-16}.

{\em{A word about notation:}} we will often drop the dependence on $x$ and even $t$ or $s$ from the unknowns $u(x,t)$ and $\eta^t(x,s)$ writing only $u$ and $\eta^t$ instead.
The norm in the space $L^p(\Omega)$ is denoted $\|\cdot\|_p$ except in the common occurrence when $p=2$ where we simply write the $L^2(\Omega)$ norm as $\|\cdot\|$. 
The $L^2(\Omega)$ product is simply denoted $(\cdot,\cdot).$
Other Sobolev norms are denoted by occurrence; in particular, since we are working with the homogeneous Dirichlet boundary conditions \eqref{bc}, in $H^1_0(\Omega)$, we will use the equivalent norm
\[
\|u\|_{H^1_0(\Omega)}=\|\nabla u\|,
\]
and in particular, 
\begin{align}
\|u\| \le \frac{1}{\sqrt{\lambda_1}}\|\nabla u\|,  \label{Poincare}
\end{align}
where $\lambda_1>0$ denotes the first eigenvalue of the Dirichlet--Laplacian.
With $D(-\Delta)=H^2(\Omega)\cap H^1_0(\Omega),$ we are able to define, for any $s\ge0,$ 
\[
H^s:=D((-\Delta)^{s/2}).
\]
Given a subset $B$ of a Banach space $X$, denote by $\|B\|_X$ the quantity $\sup_{x\in B}\|x\|_X$.
Finally, in many calculations $C$ denotes a {\em{generic}} positive constant which may or may not depend on several of the parameters involved in the formulation of the problem, and $Q(\cdot)$ will denote a {\em{generic}} positive nondecreasing function.

Concerning the model problem, we make the following assumptions.

\begin{description}
\item[(H1)] Let $a\in C^1({\overline{\Omega}})$ be such that 
\[
{\mathrm{meas}}\{ x\in\Gamma:a(x)>0 \} >0,
\]
and
\[
\mathcal{V}^1_a:=\left\{\psi\in L^2(\Omega):\int_\Omega a(x)|\nabla\psi(x)|^2 dx < \infty,\ \psi_{\mid\Gamma}=0 \right\}, 
\]
is a Hilbert space endowed with the product 
\[
(\chi,\psi)_{\mathcal{V}^1_a}:=\int_\Omega a(x)\nabla\chi(x) \cdot \nabla\psi(x) dx.
\]
(Two examples are given in \cite{CFM-16}.)
Above $\psi_{\mid\Gamma}=0$ is meant in the sense of trace which is well-defined when $\mathcal{V}^1_a\hookrightarrow W^{1,1}(\Omega)$.
In addition, we also assume the continuous embeddings hold
\[
H^{1}_0(\Omega)\hookrightarrow \mathcal{V}^1_a\hookrightarrow L^2(\Omega),
\]
and also that $Au:=\div(a(x)\nabla u)$ is a self-adjoint non-positive operator.

\item[(H2)] Assume $b\in L^\infty(\Omega)$ is a non-negative function and $c_0$ is a constant satisfying, for all $x\in\Omega,$ 
\begin{equation}
\inf_{x\in\Omega}\{a(x)+b(x)\} \ge c_0 > 0.  \label{away}
\end{equation}

\item[(H3)] Assume $g\in C^1(\mathbb{R}^+)\cap L^1(\mathbb{R}^+)$ satisfies, for all $s\ge0,$
\begin{equation}
g(s) \ge 0 \quad \text{and} \quad g'(s) \le -\delta g(s).  \label{kernel}
\end{equation}
We also impose on $g$ the smallness condition
\begin{equation}
k_0:=\int_0^\infty g(s) ds < \|a\|^{-1}_{\infty}.  \label{small-g}
\end{equation}
\end{description}

\begin{remark}
Assumption (H1) allows us to set the space for the past history function $\eta^t.$
Indeed, define 
\begin{equation}  \label{past-space}
\mathcal{M}^{0}:=L^2_g(\mathbb{R}^+;\mathcal{V}^{1}_a)=\left\{ \eta(x,s):\int_0^\infty g(s)\|\eta(x,s)\|^2_{\mathcal{V}^{1}_a}ds <\infty \right\}
\end{equation}
which is Hilbert with the product 
\[
(\eta,\zeta)_{\mathcal{M}^{0}}:=\int_0^\infty g(s) \left( \int_\Omega a(x)\nabla\eta(x,s)\cdot \nabla\zeta(x,s) dx \right) ds.
\]
\end{remark}

\begin{remark}  \label{r:degenerate}
It should be noted that in \cite{CFM-16}, the assumption (H2) allows one to view the role of the frictional damping coefficient $b$ as an arbitrarily small complementary damping in the following sense: if $\omega_0:=\{x\in\mathbb{R}^3:a(x)=0\}$, then what is only required is $b(x)>0$ on any neighborhood of $\omega_0$.
\end{remark}

\begin{remark}  \label{r:a-bnd}
Equation \eqref{kernel} of assumption (H3) implies $g$ decays to zero exponentially.
Moreover, by \eqref{small-g}, we have that, for all $x\in{\overline{\Omega}}$, 
\begin{align}
0<\ell_0 \le 1-k_0a(x)  \label{ell-ineq}
\end{align}
where
\begin{equation*}
\ell_0:=1-k_0\|a\|_\infty.
\end{equation*}
\end{remark}

Now we make our final assumptions.

\begin{description}
\item[(H4)] Let $f\in C^2(\Omega)$ and assume there exists $C_f>0$ such that, for all $s\in\mathbb{R},$ 
\begin{equation}
|f''(s)| \le C_f (1+|s|).  \label{assm-f-1}
\end{equation}
(Hence, the nonlinear term is allowed to attain critical growth.)
We also assume that 
\begin{equation}
\liminf_{|s|\rightarrow\infty} \frac{f(s)}{s} > -\ell_0\lambda_1  \label{assm-f-2}
\end{equation}
cf. \eqref{Poincare}.
\end{description}

\begin{remark}
The two conditions \eqref{assm-f-1} and \eqref{assm-f-2} are used in \cite{GGPS05} which treats the asymptotic behavior of a phase-field equation with memory.
The assumption \eqref{assm-f-1} implies there is a constant $C>0$ such that for all $r,s\in\mathbb{R}$
\begin{equation}
|f(r)-f(s)| \le C |r-s|(1+|r|^2+|s|^2).  \label{cons-f-1}
\end{equation}
The condition \eqref{cons-f-1} appears in many recent works on semilinear wave equations with memory (e.g. \cite{FePeAn2016}) and the strongly damped wave equation (this condition refers to the {\em sub}critical setting of those problems), see for example \cite{Carvalho_Cholewa_02-2,Carvalho_Cholewa_02,DellOro_Pata_2011,Graber-Shomberg-16,Pata&Squassina05,Plinio&Pata09,PPZ08}.
By \eqref{assm-f-2} we find that for some $\alpha\in(0,\lambda_1)$, there exists $\rho_f>0$ so that, for all $s\in\mathbb{R},$ there hold
\begin{equation}
f(s)s \ge -\ell_0\alpha s^2 - \rho_f  \label{cons-f-2}
\end{equation}
and, for $F(s):=\int_0^s f(\sigma)d\sigma$,
\begin{equation}
F(s) \ge -\frac{\ell_0\alpha}{2}s^2 - \rho_f.  \label{cons-f-3}
\end{equation}
Observe though both \eqref{cons-f-2} and \eqref{cons-f-3} follow when \eqref{assm-f-2} is replaced by the less general assumption,
\begin{equation} \label{cond-99}
\liminf_{|s|\rightarrow\infty} f'(s) \ge -\ell_0\lambda_1.
\end{equation}
Assumption \eqref{assm-f-1} and condition \eqref{cond-99} appear in equations with memory terms \cite{CGG11,Conti-Mola-08,CPS06,PPZ08}.
\end{remark}

Concerning the new regularity results described in section \ref{s:reg}, we additionally assume the following assumptions hold along with (H1)-(H4).

\begin{description}
\item[(H1r)] Suppose $a\in C^1({\overline{\Omega}})$ is such that 
\[
\mathcal{V}^2_a:=\left\{\psi\in L^2(\Omega):\int_\Omega a(x) \left( |\Delta\psi(x)|^2 + |\psi(x)|^2 \right) dx < \infty,\ \psi_{\mid\Gamma}=0 \right\}, 
\]
is a Hilbert space endowed with the product 
\[
(\chi,\psi)_{\mathcal{V}^2_a}:=\int_\Omega a(x) \left( \Delta\chi(x) \Delta\psi(x) + \chi(x) \psi(x) \right) dx.
\]
Also, assume the continuous embedding holds
\begin{equation*}
\mathcal{V}^2_a\hookrightarrow H^1_0(\Omega).
\end{equation*}
\end{description}

\begin{remark}
It should be noted that the embedding $D(-\Delta)\hookrightarrow \mathcal{V}^2_a$, where $D(-\Delta):=H^2(\Omega)\cap H^1_0(\Omega)$, does not hold. 
The interested reader should see \cite[Section 3]{CRV-08} where it is shown $H^2(\Omega)\not\subseteq\mathcal{V}^2_a$.
\end{remark}

\begin{description}
\item[(H4r)] Assume that there exists $\vartheta>0$ such that, for all $s\in\mathbb{R},$
\begin{equation}
f'(s)\ge-\vartheta.  \label{assm-f-3}
\end{equation}
\end{description}

\begin{remark}
The last assumption \eqref{assm-f-3} appears in \cite{CGG11,Frigeri10,Gal&Grasselli12,Gal-Shomberg15,Pata-Zelik-06}. 
Such a bound is commonly utilized to obtain the precompactness property for the semiflow associated with evolution equations where the use of fractional powers of the Laplace operator present a difficulty, if they are even well-defined.
\end{remark}

Throughout the remainder of this article, we simply denote \eqref{pde-1}-\eqref{ic} under assumptions (H1)-(H4) and (H1r) and (H4r) as problem P.

The finite energy phase-spaces we study problem P in involve the following Hilbert spaces.
First,
\[
\mathcal{H}^{0}:= H^{1}(\Omega)\times L^2(\Omega)\times \mathcal{M}^{0},
\]
endowed with the norm whose square is given by, for $U=(u,v,\eta)\in\mathcal{H}^{0},$
\[
\|U\|^2_{\mathcal{H}^{0}}:=\|\nabla u\|^2 + \|v\|^2 +\|\eta\|^2_{\mathcal{M}^{0}}.
\]
Later we also require 
\begin{equation*}
\mathcal{M}^1:=L^2_g(\mathbb{R}^+;\mathcal{V}^2_a)=\left\{ \eta:\int_0^\infty g(s)\|\eta(s)\|^2_{\mathcal{V}^2_a}ds <\infty \right\}
\end{equation*}
and
\[
\mathcal{H}^1:= H^2(\Omega)\times H^1(\Omega)\times \mathcal{M}^1,
\]
with the norm whose square is given by, for $U=(u,v,\eta)\in\mathcal{H}^1,$
\[
\|U\|^2_{\mathcal{H}^1}:=\|u\|^2_{H^{2}(\Omega)} + \|v\|^2_{H^1(\Omega)} +\|\eta\|^2_{\mathcal{M}^1}.
\]
Here $H^1(\Omega)$ is normed with 
\[
\|\psi\|_{H^1(\Omega)} = \left( \|\nabla \psi\| + \|\psi\| \right)^{1/2},  
\]
and concerning the $H^2(\Omega)$ norm above, we know by $H^2$-elliptic regularity theory (cf. e.g. \cite[section 8.4]{GT83}),
\begin{align}
\|\psi\|_{H^2(\Omega)} \le C\left( \|\Delta \psi\| + \|\psi\| \right),  \label{reg-est}
\end{align}
for some constant $C>0.$

So that we may write problem P in an operator formulation, we also define the following spaces,
\begin{align}
D(T):=\{ \eta\in\mathcal{M}^0:\eta_s\in\mathcal{M}^0,\ \eta(0)=0 \},  \notag 
\end{align}
where $\eta_s$ denotes the distributional derivative of $\eta$ and the equality $\eta(0)=0$ is meant as
\begin{equation*}
\lim_{s\rightarrow0}\|\eta(s)\| = 0,
\end{equation*}
and
\begin{align}
D(\mathcal{L}):=\left\{ U=(u,v,\eta)\in\mathcal{H}^0 \left| \begin{array}{l} v\in H^1_0(\Omega),\ \eta\in D(T), \\ \div[(1-k_0a(x))\nabla u] + \displaystyle\int_0^\infty g(s) \div[a(x)\nabla\eta(s)] ds \in L^2(\Omega) \end{array} \right. \right\},  \notag
\end{align}
to which we observe that there holds $D(\mathcal{L})\subset\mathcal{H}^1.$
On these spaces we defined the associated operators
\begin{align}
T\eta:=-\eta_s, \quad \text{for}\ \eta\in D(T),  \notag
\end{align}
and
\begin{align}
\mathcal{L}U:= \begin{pmatrix} v \\ \div[(1-k_0a(x))\nabla u] + \displaystyle\int_0^\infty g(s)\div[a(x)\nabla\eta(s)] ds - b(x)v \\ v+T\eta \end{pmatrix}, \quad \text{for}\ U\in D(\mathcal{L}).  \notag
\end{align}
For each $t\in[0,T]$, the equation 
\begin{equation}
\eta^t_t = T\eta^t + v(t)  \label{memory-2}
\end{equation}
holds as an ODE in $\mathcal{M}^0$ subject to the initial condition
\begin{equation}
\eta^0=\eta_0\in\mathcal{M}^0.  \label{memory-3}
\end{equation}
Concerning the IVP (\ref{memory-2})-(\ref{memory-3}), we have the following proposition (cf. \cite{Pata-Zucchi-2001}).

\begin{proposition}  \label{t:gen-T}
The operator $T$ with domain $D(T)$ the generator of the right-translation semigroup.
Moreover, $\eta^t$ can be explicitly represented by
\begin{equation}
\eta^t(s) = \left\{ \begin{array}{ll} u(t)-u(t-s) & \text{if}\ 0\le s\le t \\
\eta_0(s-t) + u(t) - u(0) & \text{if}\ s>t. \end{array} \right.  \label{eta}
\end{equation}
\end{proposition}

Next we define the nonlinear functional by
\begin{align}
\mathcal{F}(U):=(0,-f(u)+h,0).  \notag
\end{align}
Problem P can now be written as the abstract Cauchy problem on $\mathcal{H}^0$,
\begin{equation}  \label{acp}
\left\{ \begin{array}{ll} \displaystyle\frac{d}{dt}U = \mathcal{L}U+\mathcal{F}(U), & t>0, \\ U(0)=U_0=(u_0,u_1,\eta_0)\in\mathcal{H}^0. \end{array} \right.
\end{equation}
Later, when we are concerned with the regularity properties of problem P, we will also be interested in a more regular subspace of $\mathcal{H}^0$ (this is discussed further below).

Concerning the spaces $\mathcal{V}^1_a$ and $\mathcal{V}^2_a$ from above, it is important to note that although the injection $\mathcal{V}^1_a \hookleftarrow \mathcal{V}^2_a$ is compact, it does not follow that the injection $\mathcal{M}^0 \hookleftarrow \mathcal{M}^1$ is. 
Indeed, see \cite{Pata-Zucchi-2001} for a counterexample.
Moreover, this means the embedding $\mathcal{H}^1\hookrightarrow\mathcal{H}^0$ is not compact.
Such compactness between the ``natural phase spaces'' is essential to obtaining further regularity for the global attractors and even for the construction of finite dimensional exponential attractors. 
To alleviate this issue we follow \cite{GRP01,Pata-Zucchi-2001} (also see \cite{CPS06,GMPZ10}) and define the so-called {\em{tail function}} of $\eta\in\mathcal{M}^{0}$ by, for all $\tau\ge0,$ 
\begin{equation*}
\mathbb{T}(\tau;\eta) := \int\limits_{(0,1/\tau)\cup(\tau,\infty)} g(s) \|\nabla\eta(s)\|^2 ds.
\end{equation*}
With this we set,
\begin{equation*}
\mathcal{T}^1 := \left\{ \eta\in\mathcal{M}^1 : \eta_s\in\mathcal{M}^{0},\ \eta(0)=0,\ \sup_{\tau\ge1} \tau\mathbb{T}(\tau;\eta)<\infty \right\}.
\end{equation*}
The space $\mathcal{T}^1$ is Banach with the norm whose square is defined by
\begin{equation}
\|\eta\|^2_{\mathcal{T}^1} := \|\eta\|^2_{\mathcal{M}^1} + \|\eta_s\|^2_{\mathcal{M}^{0}} + \sup_{\tau\ge1} \tau\mathbb{T}(\tau;\eta).  \label{new-norm}
\end{equation}
Importantly, the embedding $\mathcal{T}^1\hookrightarrow\mathcal{M}^0$ is compact.
(We should mention that although the works \cite{CPS06,GMPZ10} treat PDE with an {\em{integrated past history}} variable, the compactness issue still applies to models with a relative displacement history variable, such as \eqref{memory-defn} here.
In fact, the compactness issue is more delicate in this setting; one must introduce so-called ``tail functions,'' cf. \cite[Lemma 3.1]{CPS06} or \cite[Proposition 5.4]{GMPZ10}).
Hence, let us now also define the space 
\begin{align}
\mathcal{K}^1 := H^{2}(\Omega)\times H^{1}(\Omega)\times \mathcal{T}^1, \label{reg-compact}
\end{align}
and the desired {\em{compact}} embedding $\mathcal{K}^1\hookrightarrow\mathcal{H}^0$ holds.
Again, each space is equipped with the corresponding graph norm whose square is defined by, for all $U=(u,v,\eta)\in\mathcal{K}^1$,
\begin{equation*}
\|U\|^2_{\mathcal{K}^1} := \|u\|^2_{H^2(\Omega)} + \|v\|^2_{H^1(\Omega)} + \|\eta\|^2_{\mathcal{T}^1}.  \label{reg-norm}
\end{equation*}

Concerning the IVP (\ref{memory-2})-(\ref{memory-3}), we will also call upon the following (cf. \cite[Lemmas 3.6]{CPS06}). 

\begin{lemma}  \label{what-2}
Let $\eta_0\in D(T)$.
Assume there is $\rho>0$ such that, for all $t\ge0$, $\|\nabla u(t)\|\le\rho$.
Then there is a constant $C>0$ such that, for all $t\ge0$,
\begin{align}
\sup_{\tau\ge1} \tau\mathbb{T}(\tau;\eta^t) \le 2 \left( t+2 \right)e^{-\delta t} \sup_{\tau\ge1} \tau\mathbb{T}(\tau;\eta_0) + C\rho^2.  \notag
\end{align}
\end{lemma}

We now report some results from \cite{CFM-16} who only need to assume (H1)-(H4) hold.
The following result is from \cite[Theorem 2.1]{CFM-16}.
The proof follows by relying on classical semigroup theory; namely, the operator $\mathcal{L}$ is the infinitesimal generator of a $C^0$-semigroup of contractions $e^{\mathcal{L}t}$ in $\mathcal{H}^0$ (cf. \cite[Lemma 3.1]{CFM-16}) and the local Lipschitz continuity of $\mathcal{F}:\mathcal{H}^0\rightarrow\mathcal{H}^0$.

\begin{theorem}  \label{t:well-posedness}
Given $h\in L^2(\Omega)$ and $U_0=(u_0,u_1,\eta_0)\in\mathcal{H}^0$, problem P possesses a unique global mild solution satisfying the regularity
\begin{equation}  \label{mild-reg}
u\in C([0,\infty);H^1_0(\Omega)), \quad u_t\in C([0,\infty);L^2(\Omega)) \quad \text{and} \quad \eta^t\in C([0,\infty);\mathcal{M}^0).
\end{equation}
If $U_0=(u_0,u_1,\eta_0)\in D(\mathcal{L})$, the solution is regular and  satisfies 
\begin{equation}  \label{str-reg}
U\in C([0,\infty);D(\mathcal{L})).
\end{equation}
In addition, if $Z^i(t)=(u^i(t),u^i_t(t),\eta^{i,t})$, $i=1,2$, are any two mild solutions to problem P corresponding to the initial data $Z^1_0,Z^2_0\in\mathcal{H}^0$, respectively, where $\|Z^1_0\|_{\mathcal{H}^0}\le R$ and $\|Z^2_0\|_{\mathcal{H}^0}\le R$ for some $R>0$, then for any $T>0$ and for all $t\in[0,T],$
\begin{equation}  \label{cont-dep}
\|Z^1(t)-Z^2(t)\|_{\mathcal{H}^0} \le e^{Q(R)T}\|Z^1(0)-Z^2(0)\|_{\mathcal{H}^0}
\end{equation}
for some positive nondecreasing function $Q(\cdot)$.
\end{theorem}

The next result depends on \cite[Lemma 3.3]{CFM-16}.
For this we define the ``energy functional'' which is used to extend local solutions to global ones, as well as demonstrate the gradient structure of problem P.
\begin{equation}  \label{energy}
E(t) := \|u_t(t)\|^2 + \int_\Omega (1-k_0a(x))|\nabla u(t)|^2 dx + \|\eta^t\|^2_{\mathcal{M}^0} + 2\int_\Omega \left( F(u(t))-h(x)u(t) \right)dx.
\end{equation}

\begin{lemma}
The energy $E(t)$ is non-increasing along any solution $U(t)=(u(t),u_t(t),\eta^t)$.
In addition, there exists $\delta_0,C_{fh}>0$, independent of $U$, such that for all $t\ge0,$
\begin{equation}  \label{lower-e}
E(t) \le \delta_0\|(u(t),u_t(t),\eta^t)\|^2_{\mathcal{H}^0} - C_{fh}.
\end{equation}
\end{lemma}

The following is \cite[Theorem 2.2]{CFM-16}.

\begin{theorem}  \label{t:global-attr}
Let $h\in L^2(\Omega)$ and $U_0=(u_0,u_1,\eta_0)\in\mathcal{H}^0$.
The dynamical system $(\mathcal{H}^0,S(t))$ generated by the mild solutions of Problem P is gradient and possesses a global attractor $\mathcal{A}$ which has finite (fractal) dimension and coincides with the unstable manifold $\mathbb{M}^n(\mathcal{N})$ of stationary solutions of problem P.
\end{theorem}

The final two results here will be useful in the next section.
Each result follows from the existence of a (bounded) attractor in $\mathcal{H}^0$.
The first result provides a uniform bound on the mild solutions of problem P and some extremely important dissipation integrals, and the second provides the existence of an absorbing set in a natural way.

\begin{corollary}  \label{t:unif-bnd}
For each $R>0$ and every $U_0=(u_0,u_1,\eta_0)\in\mathcal{H}^0$ such that $\|U_0\|_{\mathcal{H}^0}\le R$, there holds, for all $t\ge0,$
\begin{equation}  \label{unif-bnd}
\|S(t)U_0\|_{\mathcal{H}^0}\le Q(R)
\end{equation}
for some positive nondecreasing function $Q(\cdot).$
In addition, there exists a function $Q(\cdot)$ such that 
\begin{equation}  \label{unif-bnd-2}
\int_0^\infty \left( \|\sqrt{b(x)}u_t(\tau)\|^2 + \delta\|\eta^\tau\|^2_{\mathcal{M}^0} \right) d\tau \le Q(R).
\end{equation}
Consequently, there also holds
\begin{equation}  \label{unif-bnd-3}
\int_0^\infty \|u_t(\tau)\|^2 d\tau \le Q(R).
\end{equation}
\end{corollary}

\begin{proof}
The first result is a consequence of the existence of a global/universal attractor.

To show \eqref{unif-bnd-2}, let $R>0$ be given and $U_0\in\mathcal{H}^0$ be such that $\|U_0\|_{\mathcal{H}^0}\le R.$ 
Next we formally derive the ``energy identity'' associated with problem P by multiplying \eqref{pde-1} by $2u_t$ to then integrate over $\Omega$; this yields (cf. \cite[Equation (3.7)]{CFM-16}), 
\begin{align}
\frac{d}{dt} E + 2\int_0^\infty g(s) \int_\Omega a(x)\nabla\eta^t(s)\cdot \nabla u_t dx ds + 2\|\sqrt{b(x)}u_t\|^2 = 0.  \label{ued-1}
\end{align}
where $E$ is the energy functional \eqref{energy}.
Observe, thanks to \eqref{unif-bnd}, \eqref{ell-ineq} and \eqref{cons-f-3}, we readily find $C(R)>0$ such that, for all $t\ge0,$ 
\begin{equation}  \label{ued-4}
|E(t)| \le C(R).  
\end{equation}
Next we note that with \eqref{problem-w}$_2$ there holds,
\begin{align}
2\int_0^\infty g(s) \int_\Omega a(x)\nabla\eta^t(s)\cdot \nabla u_t dx ds & = \frac{d}{dt}\|\eta^t\|^2_{\mathcal{M}^0} + \int_0^\infty g(s) \frac{d}{ds}\|\eta^t\|^2_{\mathcal{V}^1_a} ds,  \notag
\end{align}
and applying \eqref{kernel} yields, 
\begin{align}
\int_0^\infty g(s)\frac{d}{ds}\|\eta^t(s)\|^2_{\mathcal{V}^1_a} ds & = - \int_0^\infty g'(s)\|\eta^t(s)\|^2_{\mathcal{V}^1_a} ds  \notag \\ 
& \ge \delta\int_0^\infty g(s)\|\eta^t(s)\|^2_{\mathcal{V}^1_a} ds.  \label{ued-5}
\end{align}
Hence, we have
\begin{align}
\frac{d}{dt} E + \delta\|\eta^t\|^2_{\mathcal{M}^0} + 2\|\sqrt{b(x)}u_t\|^2 \le 0.  \label{ued-7}
\end{align}
Thus, integrating \eqref{ued-7} over $(0,t)$ produces \eqref{unif-bnd-2}.

Now we show \eqref{unif-bnd-3} easily follows from \eqref{unif-bnd-2}.
Indeed, using the Mean Value Theorem for Definite Integrals, for each $\tau\ge0$, there is $\xi_\tau\in\Omega$ so that 
\[
\|\sqrt{b(x)}u_t\|^2 = \int_\Omega b(x)|u_t(\tau)|^2 dx = b(\xi_\tau)\|u_t(\tau)\|^2.
\]
Now consider
\[
\int_0^\infty b(\xi_\tau)\|u_t(\tau)\|^2 d\tau = \int_0^\infty \|\sqrt{b(x)}u_t(\tau)\|^2 d\tau,
\]
and $b(x)\not\equiv0$ on $\Omega$, then $b(\xi_\tau)>0$ for each $\tau\ge0$.
Define $b_*:=\inf_{\tau\ge0}b(\xi_\tau)>0$.
So with \eqref{unif-bnd-2} we find
\[
\int_0^\infty \|u_t(\tau)\|^2 d\tau \le \frac{1}{b_*}Q(R).
\]
The thesis \eqref{unif-bnd-3} follows with hypotheses (H5).
The proof is complete.
\end{proof}

\begin{corollary}  \label{t:abs-set}
The semigroup of solution operators $S(t)$ admits a bounded absorbing set $\mathcal{B}$ in $\mathcal{H}^0$; that is, for any subset $B\subset \mathcal{H}^0$, there exists $t_B\ge0$ (depending on $B$) such that for all $t\ge t_B$, $S(t)B\subset\mathcal{B}$.
\end{corollary}

\begin{proof}
The proof follows directly from the fact that the attractor $\mathcal{A}$ is bounded in $\mathcal{H}^0$; e.g., a ball in $\mathcal{H}^0$ of radius $\|\mathcal{A}\|_{\mathcal{H}^0}+1$ is an absorbing set in $\mathcal{H}^0$.
\end{proof}

\begin{remark}  \label{r:slow}
Unfortunately we do not know the rate of convergence of any bounded subset in $\mathcal{H}^0$ to the global attractor $\mathcal{A}.$
Moreover, there are several applications in the literature (not containing equations with degeneracies in crucial diffusion or damping terms) in which the rate of convergence of any bonded subset $B$ of $\mathcal{H}^0$ {\em{is}} exponential in the sense that there is a constant $\varpi>0$ such that for any nonempty bounded subset $B\subset\mathcal{H}^0$ and for all $t\ge0$ there holds, 
\[
{\mathrm{dist}}_{\mathcal{H}^0}(S(t)B,\mathcal{B}) \le Q(R)e^{-\varpi t}.
\]
Here, given two subsets $U$ and $V$ of a Banach space $X$, the {\em{Hausdorff semidistance}} between them is 
\[
{\mathrm{dist}}_{X}(U,V):=\sup_{u\in U}\inf_{v\in V}\|u-v\|_X.
\]
\end{remark}

\section{Regularity}  \label{s:reg}

The aim of this section, and indeed the aim of this article, is to show the existence of a smooth compact subset of $\mathcal{H}^0$ containing the global attractor $\mathcal{A}.$
This is achieved by finding a suitable subset $\mathcal{C}$ of $\mathcal{K}^1\hookrightarrow\mathcal{H}^0$; hence, $\mathcal{C}$ is compact in $\mathcal{H}^0.$
To this end we decompose the semigroup of solution operators by showing it splits into uniformly decaying to zero and uniformly compact parts.
With this we obtain asymptotic compactness for the associated semigroup of solution operators. 
The procedure requires some technical lemmas and a suitable Gr\"{o}nwall type inequality; the presentation follows \cite{Frigeri10,Gal-Shomberg15}.
The argument developed here will also be relied on to establish the existence of a compact attracting set. 
As a reminder to the reader, throughout this section (and the next) we assume the hypotheses (H1r), (H3r) and (H4r) hold in addition to (H1)-(H4).

The main result in this section is the following.

\begin{theorem}  \label{t:main-reg}
There exists a closed and bounded subset $\mathcal{C}\subset \mathcal{K}^1$ and a conatant $\omega>0$ such that for every nonempty bounded subset $B\subset \mathcal{H}^0$ and for all $t\ge0$, there holds
\begin{align}
{\mathrm{dist}}_{\mathcal{H}^0}(S(t)B,\mathcal{C}) \le Q(\|B\|_{\mathcal{H}^0})e^{-\omega t}.  \label{trans-1}
\end{align}
Consequently, the global attractor $\mathcal{A}$ (cf. Theorem \ref{t:global-attr}) is bounded in $\mathcal{K}^1$ and trajectories on $\mathcal{A}$ are regular solutions of the form
\begin{equation}  \label{regular}
u\in C([0,\infty);H^2(\Omega)), \quad u_t\in C([0,\infty);H^1(\Omega)) \quad \text{and} \quad \eta^t\in C([0,\infty);\mathcal{T}^1).
\end{equation}
\end{theorem}
\noindent The proof first requires several lemmas.

Set 
\begin{equation}  \label{psi}
\psi(s):=f(s)+\beta s \quad \text{with} \quad \beta \ge \vartheta \quad \text{so that} \quad \psi'(s) \ge 0
\end{equation}
and set $\Psi(s):=\int_0^s\psi(\sigma)d\sigma.$
(We remind the reader of \eqref{assm-f-3}.)
Let $U_0=(u_0,u_1,\eta_0)\in\mathcal{H}^0$.
Decompose \eqref{pde-1}-\eqref{ic} into the functions $v$, $w$, $\xi$ and $\zeta$ where $v+w=u$ and $\xi+\zeta=\eta$ satisfy, respectively, problem V and problem W which are given by
\begin{equation}
\left\{ \begin{array}{ll} v_{tt} - \div[(1-k_0a(x))\nabla v] - \displaystyle\int_0^\infty g(s)\div[a(x)\nabla\xi^t(s)]ds + b(x)v_t + \psi(u) - \psi(w) = 0 & \text{in}\ \Omega\times\mathbb{R}^+ \\
\xi^t_t = -\xi^t_s + v_t & \text{in}\ \Omega\times\mathbb{R}^+ \\ 
v(x,t) = 0, \quad \xi^t(x,s) = 0 & \text{on}\ \Gamma\times\mathbb{R}^+ \\ 
v(x,0) = u_0(x), \quad v_t(x,0) = u_1(x), \quad \xi^t(x,0) = 0, \quad \xi^0(x,s) = \eta_0(x,s) & \text{at}\ \Omega\times\{0\} 
\end{array}\right.  \label{problem-v}
\end{equation}
and
\begin{equation}
\left\{ \begin{array}{ll} w_{tt} - \div[(1-k_0a(x))\nabla w] - \displaystyle\int_0^\infty g(s)\div[a(x)\nabla\zeta^t(s)]ds + b(x)w_t + \psi(w) = h(x) + \beta u & \text{in}\ \Omega\times\mathbb{R}^+ \\ 
\zeta^t_t = -\zeta^t_s + w_t & \text{in}\ \Omega\times\mathbb{R}^+ \\ 
w(x,t) = 0, \quad \zeta^t(x,s) = 0 & \text{on}\ \Gamma\times\mathbb{R}^+ \\ 
w(x,0) = 0, \quad w_t(x,0) = 0, \quad \zeta^t(x,0) = 0, \quad \zeta^0(x,s) = 0 & \text{at}\ \Omega\times\{0\}.
\end{array}\right.  \label{problem-w}
\end{equation}
We now define the operators $K(t)U_0:=(w(t),w_t(t),\zeta^t)$ and $Z(t)U_0:=(v(t),v_t(t),\xi^t)$ using the associated global mild solutions to problem V and problem W (the existence of such solutions follows in a similar manor to the semigroup methods used to establish the well-posedness for problem P; cf. Theorem \ref{t:well-posedness} and the regularity described in \eqref{mild-reg}).

The first of the subsequent lemmas shows that the operators $K(t)$ are bounded bounded on $\mathcal{H}^0.$

The following lemma provides an estimate that will be extremely important later in this section.

\begin{lemma} \label{t:ride}
For each $U_0=(u_0,u_1,\eta_0)\in\mathcal{H}^0$ there exists a unique global weak solution 
\begin{equation}
W:=(w,w_t,\zeta^t)\in C([0,\infty);\mathcal{H}^0)  \label{wer-0}
\end{equation}
to problem W.
Moreover, for each $R>0$ and for all $U_0\in\mathcal{H}^0$ with $\|U_0\|_{\mathcal{H}^0}\le R$, there holds, for all $t\ge0,$
\begin{equation}  \label{wer-1}
\|K(t)U_0\|_{\mathcal{H}^0}\le Q(R)
\end{equation}
for some nonnegative increasing function $Q(\cdot).$
There also holds
\begin{equation}  \label{wer-2}
\int_0^\infty \|w_t(\tau)\|^2 d\tau \le Q(R).
\end{equation}
In addition, for every $\ep>0$ there exists a function $Q_\ep(\cdot)\sim\ep^{-1}$ such that for every $0\le s\le t$, $R>0$ and $U_0=(u_0,u_1,\eta_0)\in\mathcal{H}^0$ with $\|U_0\|_{\mathcal{H}^0}\le R,$ there holds 
\begin{align}
\int_s^t & \left( \|u_t(\tau)\|^2 + \|\sqrt{b(x)}u_t(\tau)\|^2 + \delta\|\eta^\tau\|^2_{\mathcal{M}^0} + \|w_t(\tau)\|^2 + \|\sqrt{b(x)}w_t(\tau)\|^2 + \delta\|\zeta^\tau\|^2_{\mathcal{M}^0} \right) d\tau \notag \\
& \le \frac{\ep}{2}(t-s) + Q_\ep(R).  \label{wed-0}
\end{align}
Finally, there holds
\begin{align}
\int_t^{t+1} & \left( \|u_t(\tau)\|^2 + \delta\|\eta^\tau\|^2_{\mathcal{M}^0} + \|\sqrt{b(x)}w_t(\tau)\|^2 + \|w_t(\tau)\|^2 + \delta\|\zeta^\tau\|^2_{\mathcal{M}^0} \right) d\tau \le Q(R).  \label{star}
\end{align}
\end{lemma}

\begin{proof}
As we have already stated above, the existence of global mild solutions satisfying \eqref{wer-0} follows by arguing as in the proof of Theorem \ref{t:well-posedness}.
The bound \eqref{wer-1} essentially follows from the existence of a global attractor for problem P (cf. Corollary \ref{t:unif-bnd}).
The dissipation property \eqref{wer-2} follows by arguing exactly as in the proof of Corollary \ref{t:unif-bnd} keeping in mind both $u^{(1)}$ and $u^{(b)}$ make sense, and that we are able to utilize the bound \eqref{unif-bnd-3} for either one.

We are now interested in establishing \eqref{wed-0}.
Indeed, multiplying \eqref{problem-w}$_1$ by $2w_t$ and integrating over $\Omega$, applying \eqref{problem-w}$_2$ and applying an estimate like \eqref{ued-5}, all with $w$ and $\zeta$ in place of $u$ and $\eta$, respectively, and $E_w$ denoting the corresponding functional $E$, produces (in place of \eqref{ued-7})
\begin{align}
\frac{d}{dt} & E_w + \delta\|\zeta\|^2_{\mathcal{M}^0} + 2\|\sqrt{b(x)}w_t\|^2 \le 2\beta(u,w_t).  \label{wed-1}
\end{align}
Since
\begin{align}
2\beta(u,w_t) = 2\beta(u_t,w) + 2\beta\frac{d}{dt}(u,w) \notag
\end{align}
and by \eqref{wer-1}
\begin{align}
2\beta(u_t,w) & \le \beta^2C(R)\|u_t\| \notag \\ 
& \le \ep + C_\ep\|u_t\|^2, \notag
\end{align}
so the differential inequality \eqref{wed-1} becomes
\begin{align}
\frac{d}{dt} & \{ E_w - 2\beta(u,w) \} + \delta\|\zeta^\tau\|^2_{\mathcal{M}^0} + 2\|\sqrt{b(x)}w_t\|^2  \le \ep + C_\ep\|u_t\|^2.  \label{wed-3}
\end{align}
In light of \eqref{unif-bnd-2} and \eqref{unif-bnd-3}, adding $\|u_t\|^2+\|\sqrt{b(x)}u_t(\tau)\|^2 + \delta\|\eta^\tau\|^2_{\mathcal{M}^0}$ to both sides of \eqref{wed-3} and integrating the result over $(s,t)$ then applying \eqref{unif-bnd}, \eqref{wer-1} and \eqref{ued-4} for problem W produces the desired estimate \eqref{wed-0}.

To show \eqref{star}, we now add in the bound $\|u_t\|^2+\delta\|\eta\|^2_{\mathcal{M}^0}+2\|w_t\|^2\le C(R)$ into \eqref{wed-1}, and this time estimate the right-hand side with $C(R)+\|w_t\|^2$ to obtain
\begin{align}
\frac{d}{dt} & E_w + \|u_t(\tau)\|^2 + \delta\|\eta^\tau\|^2_{\mathcal{M}^0} + \|\sqrt{b(x)}w_t(\tau)\|^2 + \|w_t(\tau)\|^2 + \delta\|\zeta^\tau\|^2_{\mathcal{M}^0} \le C(R).  \label{wed-99}
\end{align}
Integrating \eqref{wed-99} over $(t,t+1)$ and applying \eqref{ued-4} for problem W yields \eqref{star}.
\end{proof}

\begin{lemma}  \label{t:uniform-decay}
For each $U_0=(u_0,u_1,\eta_0)\in\mathcal{H}^0$ there exists a unique global weak solution 
\begin{equation}
V:=(v,v_t,\xi^t)\in C([0,\infty);\mathcal{H}^0)  \label{opp-0}
\end{equation}
to problem V.
Moreover, for each $R>0$ and for all $U_0\in\mathcal{H}^0$ with $\|U_0\|_{\mathcal{H}^0}\le R$, there exists $\omega_1>0$ such that, for all $t\geq 0$, 
\begin{equation}
\|Z(t)U_{0}\|_{\mathcal{H}^0}\le Q(R)e^{-\omega_1 t}  \label{opp-1}
\end{equation}
for some positive nondecreasing function $Q(\cdot).$
Thus, the operators $Z(t)$ are uniformly decaying to zero in $\mathcal{H}^0$.
\end{lemma}

\begin{proof}
As we have already stated above, the existence of global mild solutions satisfying \eqref{opp-0} follows by arguing as in the proof of Theorem \ref{t:well-posedness}.
It suffices to show \eqref{opp-1}.

Let $R>0$ and $U_0=(u_0,u_1,\eta_0)\in\mathcal{H}^0$ be such that $\|U_0\|_{\mathcal{H}^0}\le R.$
Next we rewrite the term $b(x)v_t$ in equation \eqref{problem-v}$_1$ as $(b(x)+1)v_t- v_t$. 
Then multiply the result in $L^2(\Omega)$ by $v_t+\ep v$, where $\ep>0$ will be chosen below.
When we include the basic identity
\begin{align}
(\psi(u)-\psi(w),v_t) & = \frac{d}{dt}\left\{(\psi(u)-\psi(w),v)-\frac{1}{2}(\psi'(u)v,v) \right\}  \notag \\
& - ((\psi'(u)-\psi'(w))w_t,v) + \frac{1}{2}(\psi''(u)u_t,v^2)  \notag
\end{align}
to the result and use \eqref{problem-v}$_2$, we find that there holds, for almost all $t\ge0,$ 
\begin{align} 
\frac{d}{dt} & \left\{ \|v_t\|^2 + 2\ep (v_t,v) + \int_\Omega (1-k_0a(x))|\nabla v|^2 dx + \|\xi^t\|^2_{\mathcal{M}^0} + \ep\|\sqrt{b(x)}v\|^2 \right.  \notag \\ 
& \left. + 2(\psi(u)-\psi(w),v) - (\psi'(u)v,v) \right\}  \notag \\ 
& - 2\ep\|v_t\|^2 + 2\ep\int_\Omega (1-k_0a(x)) |\nabla v|^2 dx - \int_0^\infty g'(s)\|\xi^t(s)\|^2_{\mathcal{V}^1_a} ds  \notag \\ 
& + 2\ep\int_0^\infty g(s) \int_\Omega a(x) \nabla\xi^t(s)\cdot\nabla v dx ds + 2\|\sqrt{b(x)}v_t\|^2 \notag \\
& - 2(\psi'(u)-\psi'(w))w_t,v) + (\psi''(u)u_t,v^2) + 2\ep(\psi(u)-\psi(w),v)  \notag \\ 
& = 0. \label{opp-2} 
\end{align}

We now consider the functional defined by 
\begin{align} 
\mathbb{V}(t) & := \|v_t(t)\|^2 + 2\ep (v_t(t),v(t)) + \int_\Omega (1-k_0a(x))|\nabla v(t)|^2 dx + \|\xi^t\|^2_{\mathcal{M}^0} + \ep\|\sqrt{b(x)}v(t)\|^2  \notag \\ 
& + 2(\psi(u(t))-\psi(w(t)),v(t)) - (\psi'(u(t))v(t),v(t))  \label{opp-3}
\end{align}
We now will show that, given $U(t)=(u(t),u_{t}(t),\eta^t), W(t)=(w(t),w_{t}(t),\zeta^t)\in \mathcal{H}^0$ are uniformly bounded with respect to $t\ge0$ by some $R>0$, there are constants $C_1,C_2>0$, independent of $t$, in which for all $V(t)=(v(t),v_t(t),\xi^t)\in \mathcal{H}^0$,
\begin{equation}
C_1\|V(t)\|_{\mathcal{H}^0}^{2} \le \mathbb{V}(t) \le C_2\|V(t)\|_{\mathcal{H}^0}^{2}.  \label{opp-4}
\end{equation}
To this end we begin by estimating the following product with \eqref{Poincare},
\begin{align}
2\ep |(v_t,v)| & \le \ep\|v_t\|^2 + \ep\|v\|^2  \notag \\ 
& \le \ep\|v_t\|^2 + \frac{\ep}{\lambda_1}\|\nabla v\|^2,  \label{opp-4.5}
\end{align}
and
\begin{align}
\ep \|\sqrt{b(x)}v\|^2 & \le \ep\|\sqrt{b}\|^2_\infty\|v\|^2  \notag \\ 
& \le \frac{\ep}{\lambda_1}\|b\|_\infty\|\nabla v\|^2.  \label{opp-4.6}
\end{align}
Concerning the terms in the functional $\mathbb{V}$ that involve the nonlinear term $\psi$, using \eqref{psi}, \eqref{assm-f-1}, \eqref{assm-f-2} and the embedding $H^{1}(\Omega)\hookrightarrow L^{6}(\Omega)$, and also \eqref{unif-bnd}, there holds
\begin{align} 
|(\psi ^{\prime }(u)v,v) | & \leq C\left( 1+\|\nabla u\|^{2}\right) \|\nabla v\|\|v\|  \notag \\ 
& \le \ep\|\nabla v\|^{2} + C_\ep(R)\|v\|^{2},  \label{opp-5} 
\end{align}
where the constant $0<C_\ep\sim\ep^{-1}.$
From assumption \eqref{assm-f-3} and \eqref{psi}
\begin{equation}
2(\psi(u)-\psi(w),v) \geq 2(\beta -\vartheta )\|v\|^{2}.  \label{opp-6}
\end{equation}%
Hence, for $\beta=\beta(\ep)$ sufficiently large, the combination of (\ref{opp-5}) and (\ref{opp-6}) produces, 
\begin{align} 
2(\psi(u)-\psi(w),v) - (\psi'(u)v,v) & \ge 2(\beta-\vartheta)\|v\|^2 - \ep\|\nabla v\|^2 - C_\ep(R)\|v\|^2  \notag \\
& \geq - \ep\|\nabla v\|^2.  \label{opp-7}
\end{align}
With \eqref{opp-4.5}, \eqref{opp-4.6} and \eqref{opp-7} we attain the lower bound for the functional $\mathbb{V}$,
\begin{equation*}
\mathbb{V} \ge \left( \ell_0 - \frac{\ep}{\lambda_1}(2+\|b\|_\infty) - \ep\right) \|\nabla v\|^2 + \left( 1-\ep \right)\|v_t\|^2 + \|\xi^t\|^2_{\mathcal{M}^0}.  \notag
\end{equation*}
So for a sufficiently small $\ep>0$ fixed (which also fixes the choice of $\beta$), there is $m_0>0$ in which, for all $t\geq 0$, we have that 
\begin{equation}
\mathbb{V}(t) \ge m_0\|(v(t),v_{t}(t),\xi^t)\|_{\mathcal{H}^0}^{2}.  \label{opp-8}
\end{equation}
Now by the (local) Lipschitz continuity of $f$, the embedding $H^1_0(\Omega)\hookrightarrow L^2(\Omega)$, the uniform bounds on $u$ and $w$, and the Poincar\'{e} inequality \eqref{Poincare}, it is easy to check that with \eqref{cons-f-1} there holds
\begin{align}
2(\psi(u)-\psi(w),v) & \le 2\|\psi(u)-\psi(w)\| \|v\|  \notag \\
& \le C(R)\|\nabla v\|^{2}.  \label{opp-8.5}
\end{align}
Also, using \eqref{psi}, (\ref{assm-f-1}), (\ref{assm-f-2}) and the bound (\ref{unif-bnd}), there also holds
\begin{equation}
|(\psi'(u)v,v)| \leq C(R)\|\nabla v\|^{2}.  \label{opp-9}
\end{equation}
Thus, with \eqref{opp-8.5}, \eqref{opp-9} and referring to some of the above estimates, the right-hand side of (\ref{opp-4}) also follows. 

Moving forward, we now work on \eqref{opp-2}.
In light of the estimates 
\begin{align}
2|((\psi'(u)-\psi'(w))w_{t},v)| & \le C( 1 + \|\nabla u\| + \|\nabla w\|) \|w_{t}\| \|v\|^2  \notag \\
& \le \frac{1}{2\beta}\|v\|^{2} + C(R)\|w_{t}\|^{2}\mathbb{V},  \label{opp-11}
\end{align}
and
\begin{align}
|(\psi''(u)u_{t},v^2)| & \le C ( 1 + \|\nabla u\| ) \|u_{t}\| \|v\|^{2}  \notag \\
& \le \frac{1}{2\beta}\|v\|^{2} + C(R)\|u_{t}\|^{2} \mathbb{V},  \label{opp-12}
\end{align}
(here the constants $C(R)>0$ also depend on $\beta>0$) we see that with \eqref{opp-11}, \eqref{opp-12}, as well as \eqref{ell-ineq}, \eqref{kernel} and \eqref{opp-6}, the differential identity (\ref{opp-2}) becomes 
\begin{align} 
\frac{d}{dt} & \mathbb{V} + \ep\|v_t\|^2 + 2\ep\ell_0 \|\nabla v\|^2 + \delta \|\xi^t\|^2_{\mathcal{M}^0}  \notag \\ 
& + 2\ep \int_0^\infty g(s) \int_\Omega a(x) \nabla \xi^t(s)\cdot\nabla v dx ds + 2\|\sqrt{b(x)}v_t\|^2 + \left( 2\ep(\beta-\vartheta) - \frac{1}{\beta} \right)\|v\|^2  \notag \\ 
& \le C(R) \left( \|u_t\|^2 + \|w_t\|^2 \right) \mathbb{V} + 3\ep\mathbb{V},  \label{opp-14} 
\end{align}
where we also added $3\ep\|v_t\|^2$ to both sides (observe, $3\ep\|v_t\|^2\le 3\ep\mathbb{V}$).
We now seek a suitable control on the product
\begin{align}
\left| 2\ep \int_0^\infty g(s) \int_\Omega a(x) \nabla \xi^t(s)\cdot\nabla v dx ds \right| & \le 2\ep \int_0^\infty g(s) \left| \int_\Omega a(x) \nabla \xi^t(s)\cdot\nabla v dx \right| ds   \notag \\ 
& = 2\ep \int_0^\infty g(s) \left| (\xi^t(s),v)_{\mathcal{V}^1_a} \right| ds   \notag \\ 
& \le 2\ep \|\xi^t\|_{\mathcal{M}^0} \|\nabla v\|  \notag \\ 
& \le 2\sqrt{\ep}\|\xi^t\|^2_{\mathcal{M}^0} + \frac{\ep\sqrt{\ep}}{2}\|\nabla v\|^2.  \label{opp-14.3}
\end{align}
For sufficiently large $\beta>0$, we may omit the positive terms $2\|\sqrt{b(x)}v_t\|^2+(2\ep(\beta-\vartheta)-\frac{1}{\beta})\|v\|^2$ from the left-hand side of  \eqref{opp-14} so that it becomes, with \eqref{opp-14.3}, 
\begin{align} 
\frac{d}{dt} & \mathbb{V} + \ep\|v_t\|^2 + \ep\left( 2\ell_0-\frac{\sqrt{\ep}}{2} \right) \|\nabla v\|^2 + \left( \delta - 2\sqrt{\ep} \right) \|\xi^t\|^2_{\mathcal{M}^0}  \notag \\ 
& \le C(R) \left( \|u_t\|^2 + \|w_t\|^2 + 3\ep \right) \mathbb{V}.  \label{opp-14.5} 
\end{align}
For any $\ep>0$ sufficiently small so that 
\[
2\ell_0-\frac{\sqrt{\ep}}{2} >0 \quad \text{and} \quad \delta - 2\sqrt{\ep}>0,
\]
we can find a constant $m_1>0$, thanks to \eqref{opp-4}, such that \eqref{opp-14.5} can be written as the following differential inequality, to hold for almost all $t\ge0,$
\begin{align} 
\frac{d}{dt} & \mathbb{V} + \ep m_1\mathbb{V} \le C(R) \left( \|u_t\|^2 + \|w_t\|^2 + 3\ep \right) \mathbb{V}.  \label{opp-15} 
\end{align}
Here we recall Proposition \ref{GL} and Lemma \ref{t:ride}.
Applying these to (\ref{opp-15}) yields, for all $t\ge0,$
\begin{equation}
\mathbb{V}(t) \le \mathbb{V}(0) e^{Q(R)}e^{-m_1 t/2},  \label{opp-16}
\end{equation}
for some positive nondecreasing function $Q(\cdot).$
By virtue of (\ref{opp-4}) and the initial conditions provided in \eqref{problem-v}, 
\begin{align}
\mathbb{V}(0) & \le C_{2}(R) \|(v(0),v_t(0),\xi^0)\|_{\mathcal{H}^0}^{2}  \notag \\ 
& \le C_{2}(R) \left( \|\nabla u_0\|^2 + \|u_1\|^2 + \|\eta_0\|^2_{\mathcal{M}^0} \right)  \notag \\ 
& \le Q(R).  \notag
\end{align}
Therefore (\ref{opp-16}) shows that the operators $Z(t)$ are uniformly decaying to zero.
The proof is finished.
{\bf $<<<<<$}
\end{proof}

The remaining lemmas will show that the operators $K(t)$ are asymptotically compact on $\mathcal{H}_0$.
In order to establish this, we prove that the operators $K(t)$ are uniformly bounded in $\mathcal{K}^1\hookrightarrow\mathcal{H}^0.$

Due to the nature of the proof of the following lemma, we also need to assign the past history for the term $w_t$.
Indeed, from below we need to consider the initial condition
\[
\zeta^0_t(x,s) = -\zeta^0_s(x,s) = -w_t(x,0-s).
\]
However, since $u=v+w$, we can write
\[
-u_t(x,0-s) = -v_t(x,0-s) - w_t(x,0-s)
\]
and hence assume that 
\begin{align}
v_t(x,0-s) = u_t(x,0-s) = -\eta^0_t(x,s) \quad \text{and} \quad w_t(x,0-s) = 0.  \label{initial}
\end{align}

\begin{lemma}  \label{t:diff-bnd}
For each $R>0$ and for all $U_{0}=(u_{0},u_{1},\eta_0)\in \mathcal{H}^0$ such that $\|U_{0}\|_{\mathcal{H}^0} \le R$, there holds for all $t\ge0$
\begin{align}
\|\partial_tK(t)U_{0}\|^2_{\mathcal{H}^0} = \|\nabla w_t(t)\|^2 + \|w_{tt}(t)\|^2 + \|\zeta^t_t\|^2_{\mathcal{M}^0} \le Q(R)  \label{lpp-12}
\end{align}
for some positive nondecreasing function $Q(\cdot)$. 
\end{lemma}

\begin{proof}
For all $x\in\Omega$ and $t,s\in\mathbb{R}^+$, set $H(x,t):=w_t(x,t)$ and $X^t:=\zeta^t_t(s).$
Differentiating problem W with respect to $t$ yields the system
\begin{equation}  \label{problem-H}
\left\{ \begin{array}{ll} H_{tt} - \div[(1-k_0a(x))\nabla H] - \displaystyle\int_0^\infty g(s)\div[a(x)\nabla X^t(s)]ds + b(x)H_t + \psi'(w)H = \beta u_t & \text{in}\ \Omega\times\mathbb{R}^+ \\  
X^t_t = -X^t_s + H_t & \text{in}\ \Omega\times\mathbb{R}^+ \\ 
H(x,t) = w_t(x,t) = 0, \quad X^t(x,s) = \zeta^t_t(x,s) & \text{on}\ \Gamma\times\mathbb{R}^+ \\ 
H(x,0) = w_t(x,0) = 0, \quad H_t(x,0) = w_{tt}(x,0) = -f(0)-u_1 \quad \text{(from \eqref{problem-w})} & \text{at}\ \Omega\times\{0\} \\ 
X^t(x,0) = w_t(x,t) - w_t(x,t-0) = 0, \quad X^0(x,s) = 0 \quad \text{(see \eqref{initial})} & \text{at}\ \Omega\times\{0\}.
\end{array}\right.
\end{equation}
Multiply equation \eqref{problem-H}$_1$ by $H_t+\ep H$ for some $\ep>0$ to be chosen below.
To this result we apply the identities 
\[
(\psi'(w)H,H_t) = \frac{1}{2}\frac{d}{dt}(\psi'(w)H,H) - \frac{1}{2}(\psi''(w)w_t,H^2),
\]
and (here we rely on \eqref{problem-H}$_2$)
\begin{align}
\int_0^\infty g(s) & \int_\Omega a(x)\nabla X^t(s)\nabla H_t(t) dxds  \notag \\ 
& = \frac{1}{2}\frac{d}{dt}\|X^t\|^2_{\mathcal{M}^0} + \int_0^\infty g(s)\frac{d}{ds}\|X^t(s)\|^2_{\mathcal{V}^1_a}ds  \notag \\
& = \frac{1}{2}\frac{d}{dt}\|X^t\|^2_{\mathcal{M}^0} - \int_0^\infty g'(s)\|X^t(s)\|^2_{\mathcal{V}^1_a}ds  \notag
\end{align}
so that together we find 
\begin{align}
& \frac{d}{dt} \left\{ \|H_t\|^2 + 2\ep(H_t,H) + \int_\Omega (1-k_0a(x))|\nabla H|^2dx + \|X^t\|^2_{\mathcal{M}^0} + (\psi'(w)H,H) \right\}  \notag \\ 
& - 2\ep\|H_t\|^2 + 2\ep\|\sqrt{b(x)}H_t\|^2 + 2\ep(b(x)H_t,H) + 2\ep\int_\Omega (1-k_0a(x))|\nabla H|^2dx + 2\ep(\psi'(w)H,H)  \notag \\ 
& - 2\int_0^\infty g'(s)\|X^t(s)\|^2_{\mathcal{V}^1_a}ds + 2\ep\int_0^\infty g(s)\int_\Omega a(x)\nabla X^t(s)\cdot \nabla H(t) dx ds  \notag \\
& = (\psi''(w)w_t,H^2) + 2\beta(u_t,H_t) + 2\beta\ep(u_t,H).  \label{lpp-1}
\end{align}
Next we recall \eqref{kernel} and find
\begin{align}
-2\int_0^\infty g'(s)\|X^t(s)\|^2_{\mathcal{V}^1_a} ds & \ge 2\delta\|X^t\|^2_{\mathcal{M}^0},  \label{lpp-2}
\end{align}
and 
\begin{align}
2\ep\int_0^\infty g(s)\int_\Omega a(x)\nabla X^t(s)\cdot \nabla H(t) dx ds & \ge -\delta\|X^t\|^2_{\mathcal{M}^0} - \frac{\ep^2}{\delta}\|\nabla H\|^2,  \label{lpp-3}
\end{align}
where the last inequality follows from \eqref{small-g}.
For all $\ep>0$ and $t\ge0$, define the functional 
\begin{align}
\mathbb{I}(t):=\|H_t(t)\|^2 + 2\ep(H_t(t),H(t)) + \int_\Omega (1-k_0a(x))|\nabla H(t)|^2dx + \|X^t\|^2_{\mathcal{M}^0} + (\psi'(w)H(t),H(t)).  \label{lpp-4}
\end{align}
Thanks to \eqref{ell-ineq} and since $\psi'>0$, there is a constant $C>0$, sufficiently small, so that
\begin{align}
C\left( \|H_t(t)\|^2 + \ell_0\|\nabla H(t)\|^2 + \|X^t\|^2_{\mathcal{M}^0} \right) \le \mathbb{I}(t).  \label{lpp-4.5}
\end{align}
At this point we can write \eqref{lpp-1}-\eqref{lpp-3} with \eqref{lpp-4} as
\begin{align}
& \frac{d}{dt}\mathbb{I} - 2\ep\|H_t\|^2 + 2\ep\|\sqrt{b(x)}H_t\|^2 + 2\ep(b(x)H_t,H) + \left( 2\ep\ell_0 - \frac{\ep^2}{\delta} \right)\|\nabla H\|^2  \notag \\ 
& + \delta\|X^t\|^2_{\mathcal{M}^0} + 2\ep(\psi'(w)H,H)  \notag \\
& \le 2(\psi''(w)w_t,H^2) + 2\beta(u_t,H_t) + 2\beta\ep(u_t,H).  \label{lpp-5}
\end{align}
Next, let us rely on the uniform bounds \eqref{unif-bnd} and \eqref{opp-1} to estimate the products on the right-hand side
\begin{align}
2|(\psi''(w)w_t,H^2)| & \le 2\|\psi''(w)w_t H^2\|_{1}  \notag \\ 
& \le 2\|\psi''(w)w_t\|_{3/2}\|H\|^2_6  \notag \\
& \le 2\|\psi''(w)\|_{6}\|w_t\|\|H\|^2_6  \notag \\ 
& \le C(R)\|w_t\|\|\nabla H\|^2  \notag \\ 
& \le C(R)\|w_t\|\mathbb{I},  \label{lpp-7}
\end{align}
\begin{align}
2\beta|(u_t,H_t)+\ep(u_t,H)| & \le C(R)\|H_t\| + C(R)\|\nabla H\|  \notag \\ 
& \le C_\ep(R) + \ep\|H_t\|^2 + \ep^2\|\nabla H\|^2,  \label{lpp-8}
\end{align}
where $C_\ep\sim\ep^{-1}\wedge\ep^{-2}$.
Also, we know
\begin{align}
2\ep(\psi'(w)H,H) & \ge 2\ep^2(\beta-\vartheta)\|H\|^2 >0.  \label{lpp-8.5}
\end{align}
Thus, combining \eqref{lpp-5}-\eqref{lpp-8.5} yields 
\begin{align}
\frac{d}{dt}\mathbb{I} & - 3\ep\|H_t\|^2 + 2\ep\|\sqrt{b(x)}H_t\|^2 + \ep\left( 2\ell_0-\ep\left(\frac{1}{\delta}+1\right)\right)\|\nabla H\|^2 + \delta\|X^t\|^2_{\mathcal{M}^0}  \notag \\
& \le C(R)\|w_t\|\mathbb{I} + C_\ep(R).  \label{lpp-10}
\end{align}
Since $4\ep\|H_t\|^2\le4\ep\mathbb{I}$, adding this to \eqref{lpp-10} makes the differential inequality (we also omit $2\ep\|\sqrt{b(x)}H_t\|^2$)
\begin{align}
\frac{d}{dt}\mathbb{I} & + \ep\|H_t\|^2 + \ep\left( 2\ell_0-\ep\left(\frac{1}{\delta}+1\right)\right)\|\nabla H\|^2 + \delta\|X^t\|^2_{\mathcal{M}^0}  \notag \\
& \le C(R)\left( \|w_t\| + \ep \right)\mathbb{I} + C_\ep(R).  \label{lpp-10.5}
\end{align}
We now find that for any $\ep>0$ small so that
\[
2\ell_0-\ep\left(\frac{1}{\delta}+1\right) >0,
\]
then 
\begin{align}
& \frac{d}{dt}\mathbb{I} + \ep\mathbb{I} \le C(R)\left( \|w_t\| + \ep \right)\mathbb{I} + C_\ep(R) \label{lpp-11}
\end{align}
to which we now apply Proposition \ref{GL2} and the bounds \eqref{wed-0} and \eqref{star} to conclude that, for all $t\ge0$, there holds
\begin{align}
\mathbb{I}(t) \le C(R) \mathbb{I}(0)e^{-\ep t/2} + C_\ep(R).  \label{lpp-11.5} 
\end{align}
Moreover, with \eqref{lpp-4} and the initial conditions in \eqref{problem-H} we find that there is a constant $C>0$ (with $\ep>0$ now fixed) in which 
\begin{align}
\|H_t(t)\|^2 + \|\nabla H(t)\|^2 + \|X^t\|^2_{\mathcal{M}^0} & \le C(R).  \notag
\end{align}
This establishes \eqref{lpp-12} and completes the proof.
\end{proof}

We derive the immediate consequence of \eqref{problem-w} and \eqref{lpp-12}.

\begin{corollary}  \label{t:part-z}
Under the assumptions of Lemma \ref{t:diff-bnd}, there holds for all $t\ge0,$
\begin{align}
\|\zeta^t_s\|_{\mathcal{M}^0} \le Q(R).  \label{part-z}
\end{align}
\end{corollary}

Before we continue, we derive a further estimate for $\zeta^t$.

\begin{lemma}
Under the assumptions of Lemma \ref{t:diff-bnd}, there holds for all $t\ge0$,
\begin{align}
\|\nabla\zeta^t\|_{L^2_g(\mathbb{R}^+;L^2(\Omega))} \le C_\delta.  \label{ccc-2}
\end{align}
\end{lemma}

\begin{proof}
Formally multiplying \eqref{problem-w}$_2$ in $L^2_g(\mathbb{R}^+;L^2(\Omega))$ by $-\Delta \zeta^t(s)$ and estimating the result yields the differential inequality
\begin{align}
\frac{d}{dt}\|\nabla\zeta^t\|^2_{L^2_g(\mathbb{R}^+;L^2(\Omega))} & = -\int_0^\infty g(s) \frac{d}{ds} \|\nabla\zeta^t(s)\|^2 ds + (\nabla w_t,\nabla\zeta^t)_{L^2_g(\mathbb{R}^+;L^2(\Omega))}  \notag \\ 
& = \int_0^\infty g'(s) \|\nabla\zeta^t(s)\|^2 ds + (\nabla w_t,\nabla\zeta^t)_{L^2_g(\mathbb{R}^+;L^2(\Omega))}  \notag \\  
& \le -\delta\int_0^\infty g(s) \|\nabla\zeta^t(s)\|^2 ds + \frac{2}{\delta}\|\nabla w_t\|^2 + \frac{\delta}{2}\|\nabla\zeta^t\|^2_{L^2_g(\mathbb{R}^+;L^2(\Omega))}  \notag \\  
& = -\frac{\delta}{2}\|\nabla\zeta^t\|^2_{L^2_g(\mathbb{R}^+;L^2(\Omega))} + \frac{2}{\delta}\|\nabla w_t\|^2.  \label{ccc-1}
\end{align}
Hence, applying the bound \eqref{lpp-12} to \eqref{ccc-1}, we find the differential inequality which holds for almost all $t\ge0$
\begin{align}
\frac{d}{dt}\|\nabla\zeta^t\|^2_{L^2_g(\mathbb{R}^+;L^2(\Omega))} + \frac{\delta}{2}\|\nabla\zeta^t\|^2_{L^2_g(\mathbb{R}^+;L^2(\Omega))} \le C_\delta  \notag
\end{align}
where $0<C_\delta\sim\delta^{-1}.$
Applying a straight-forward Gr\"{o}nwall inequality and the initial conditions in \eqref{problem-w} produces the desired bound \eqref{ccc-2}.
This concludes the proof. 
\end{proof}

\begin{lemma}  \label{t:uniform-compactness}
Under the assumptions of Lemma \ref{t:diff-bnd}, the following holds for all $t>0$,
\begin{equation}
\|K(t)U_{0}\|_{\mathcal{K}^1} \le Q(R),  \label{compact}
\end{equation}
for some positive nondecreasing function $Q(\cdot)$. 
Furthermore, the operators $K(t)$ are uniformly compact in $\mathcal{H}^0$.
\end{lemma}

\begin{proof}
The proof consists of several  parts. 
In the first part, we derive further bounds for some higher order terms.
We begin by rewriting/expanding \eqref{problem-w} as
\begin{align}
& w_{tt} +k_0\nabla a(x)\cdot\nabla w +(1-k_0a(x))(-\Delta)w  \notag \\
& - \displaystyle\int_0^\infty g(s)\nabla a(x)\cdot\nabla\zeta^t(s)ds + \displaystyle\int_0^\infty g(s)a(x)(-\Delta)\zeta^t(s)ds + b(x)w_t + \psi(w) = \beta u.  \label{qww-0.1}
\end{align}
Next, Using the relative displacement history definition of the memory space term
\begin{equation}  \label{memory-w}
\zeta^t(s) := w(x,t) - w(x,t-s),
\end{equation}
we rewrite the integral 
\begin{align}
\displaystyle\int_0^\infty g(s)a(x)(-\Delta)\zeta^t(s)ds = k_0 a(x)(-\Delta)w - \int_0^\infty g(s)a(x)(-\Delta)w(t-s)ds.  \label{qww-0.2}
\end{align}
Combining \eqref{qww-0.1} and \eqref{qww-0.2} shows \eqref{problem-w} takes the useful alternate form
\begin{align}
& w_{tt} - \Delta w - \int_0^\infty g(s) a(x) (-\Delta) w(t-s) ds + b(x)w_t + \psi(w)  \notag \\ 
& + k_0 \nabla a(x)\cdot\nabla w - \int_0^\infty g(s) \nabla a(x)\cdot \nabla\zeta^t(s) ds = \beta u.  \label{qww-1}
\end{align}We now report six identities that will be used below:
\begin{align}
(w_{tt},(-\Delta)w) = \frac{d}{dt}(\nabla w_t,\nabla w) - \|\nabla w_t\|^2,  \label{qww-1.05}
\end{align}
\begin{align}
-\int_0^\infty & g(s)(a(x)(-\Delta)\underbrace{w(t-s)}_{=w(t)-\zeta^t(s)},(-\Delta)w_t(t)) ds  \notag \\
& = -\int_0^\infty g(s)(a(x)(-\Delta)w(t),(-\Delta)w_t(t)) ds + \int_0^\infty g(s)(a(x)(-\Delta)\zeta^t(s),(-\Delta)\underbrace{w_t(t)}_{=\zeta^t_t(s)+\zeta^t_s(s)})ds  \notag \\ 
& = -\frac{k_0}{2}\frac{d}{dt}\|w\|^2_{\mathcal{V}^2_a} + \frac{1}{2}\frac{d}{dt}\|\zeta^t\|^2_{\mathcal{M}^1} + \frac{1}{2}\int_0^\infty g(s)\frac{d}{ds}\|\zeta^t(s)\|^2_{\mathcal{V}^2_a} ds,  \label{qww-1.1}
\end{align}
\begin{align}
-\int_0^\infty & g(s)(a(x)(-\Delta)\underbrace{w(t-s)}_{=w(t)-\zeta^t(s)},(-\Delta)w(t)) ds = -k_0\|w\|^2_{\mathcal{V}^2_a} + \int_0^\infty g(s)(a(x)(-\Delta)\zeta^t(s),(-\Delta)w(t)) ds,  \label{qww-1.2}
\end{align}
\begin{align}
(b(x)w_t,(-\Delta)w_t) = \frac{d}{dt}(b(x)w_t,(-\Delta)w) - (b(x)w_{tt},(-\Delta)w),  \label{qww-1.3}
\end{align}
\begin{align}
k_0(\nabla a(x)\cdot \nabla w,(-\Delta)w_t) = \frac{d}{dt}k_0(\nabla a(x)\cdot \nabla w,(-\Delta)w) - k_0(\nabla a(x)\cdot \nabla w_{t},(-\Delta)w),  \label{qww-1.4}
\end{align}
and
\begin{align}
-\int_0^\infty & g(s)(\nabla a(x)\cdot \nabla\zeta^t(s),(-\Delta)w_t(t)) ds  \notag \\ 
& = -\frac{d}{dt} \int_0^\infty g(s)(\nabla a(x)\cdot \nabla\zeta^t(s),(-\Delta)w(t)) ds + \int_0^\infty g(s)(\nabla a(x)\cdot \nabla\zeta^t_t(s),(-\Delta)w(t)) ds.  \label{qww-1.5}
\end{align}
Next we multiply \eqref{qww-1} in $L^2(\Omega)$ by $(-\Delta)w_t+(-\Delta)w$ to obtain, in light of \eqref{qww-1.1}-\eqref{qww-1.5}, the differential identity 
\begin{align}
& \frac{d}{dt}\left\{ \|\nabla w_t\|^2 + 2(\nabla w_t,\nabla w) + \|\Delta w\|^2 - k_0\|w\|^2_{\mathcal{V}^2_a} + \|\zeta^t\|^2_{\mathcal{M}^1} \right.  \notag \\ 
& \left. + 2(b(x)w_t,(-\Delta)w) + 2k_0(\nabla a(x)\cdot\nabla w,(-\Delta)w) - 2\int_0^\infty g(s)(\nabla a(x)\cdot\nabla\zeta^t(s),(-\Delta)w(t)) ds \right\}  \notag \\ 
& - 2\|\nabla w_t\|^2 + 2\|\Delta w\|^2 + \int_0^\infty g(s) \frac{d}{ds}\|\zeta^t(s)\|^2_{\mathcal{V}^2_a} ds - 2k_0\|w\|^2_{\mathcal{V}^2_a}  \notag \\ 
& + 2\int_0^\infty g(s)(a(x)(-\Delta)\zeta^t(s),(-\Delta)w(t)) ds - 2(b(x)w_{tt},(-\Delta)w) + 2(b(x)w_t,(-\Delta)w)  \notag \\ 
& + 2(\psi'(w)\nabla w,\nabla w_t) + 2(\psi(w),(-\Delta)w) - 2k_0(\nabla a(x)\cdot\nabla w_t,(-\Delta)w) + 2k_0(\nabla a(x)\cdot\nabla w,(-\Delta)w)  \notag \\
& + 2\int_0^\infty g(s)(\nabla a(x)\cdot\nabla\zeta^t_t(s),(-\Delta)w(t)) ds - 2\int_0^\infty g(s)(\nabla a(x)\cdot\nabla\zeta^t(s),(-\Delta)w(t)) ds  \notag \\
& = 2\beta(\nabla u,\nabla w_t) + 2\beta(u,(-\Delta)w).  \label{qww-2}
\end{align}

We now seek a constant $m_2>0$ sufficiently small so that we can write the above differential identity in the following form
\begin{equation}  \label{qww-3}
\frac{d}{dt}\Phi + c m_2\Phi \le Q(R)
\end{equation}
where 
\begin{align}
& \Phi(t) := \|\nabla w_t(t)\|^2 + 2(\nabla w_t(t),\nabla w(t)) + \|\Delta w(t)\|^2 - k_0\|w(t)\|^2_{\mathcal{V}^2_a} + \|\zeta^t\|^2_{\mathcal{M}^1}  \notag \\ 
& + 2(b(x)w_t(t),(-\Delta)w(t)) + 2k_0(\nabla a(x)\cdot\nabla w(t),(-\Delta)w(t)) - 2\int_0^\infty g(s)(\nabla a(x)\cdot\nabla\zeta^t(s),(-\Delta)w(t)) ds.  \label{qww-3.2}
\end{align}
The important lower bound holds
\begin{align}
\Phi \ge C_1(\|\Delta w\|^2 + \|\nabla w_t\|^2 + \|\zeta^t\|^2_{\mathcal{M}^1}) - C_2(R)  \label{qww-3.5}
\end{align}
for some constants $C_1,C_2(R)>0,$ and essentially follows from some basic estimates, the bounds \eqref{unif-bnd}, \eqref{opp-1}, \eqref{lpp-12}, \eqref{ccc-2}, the Poincar\'{e} inequality \eqref{Poincare} and with the assumptions on the functions $a$ and $b$.
Indeed, we estimate, for all $\ep>0,$
\begin{align}
2|(\nabla w_t,\nabla w)| & \le \ep\|\nabla w_t\|^2 + \frac{1}{\ep}\|\nabla w\|^2  \notag \\
& \le \ep\|\nabla w_t\|^2 + C_\ep(R),  \label{qww-3.51}
\end{align}
\begin{align}
- k_0\|w\|^2_{\mathcal{V}^2_a} & = -k_0\int_\Omega a(x)|\Delta w|^2 dx  \notag \\ 
& \ge -k_0 \|a\|_\infty\|\Delta w\|^2,  \label{qww-3.52}
\end{align}
\begin{align}
2|(b(x)w_t,(-\Delta)w)| & \le \frac{1}{\ep}\|b(x)w_t\| + \ep\|\Delta w\|^2  \notag \\ 
& \le C_\ep(R) + \ep\|\Delta w\|^2,  \label{qww-3.53}
\end{align}
\begin{align}
2k_0|(\nabla a(x)\cdot\nabla w,(-\Delta)w)| & \le \frac{k_0^2}{\ep}\|\nabla a(x)\cdot\nabla w\|^2 + \ep\|\Delta w\|^2  \notag \\ 
& \le C_\ep(R) + \ep\|\Delta w\|^2,  \label{qww-3.54}
\end{align}
and
\begin{align}
2\int_0^\infty g(s)|(\nabla a(x)\cdot\nabla\zeta^t(s),(-\Delta)w(t))| ds & \le \int_0^\infty g(s)\left( \frac{1}{\ep}\|\nabla a(x)\cdot\nabla\zeta^t(s)\|^2 + \ep\|\Delta w(t)\|^2 \right) ds  \notag \\
& \le \frac{1}{\ep}\int_0^\infty g(s) \|\nabla a\|^2_\infty\|\nabla\zeta^t(s)\|^2 ds + \ep\int_0^\infty g(s)\|\Delta w(t)\|^2 ds  \notag \\
& \le \frac{1}{\ep}\|\nabla a\|^2_\infty\|\nabla\zeta^t\|^2_{L^2_g(\mathbb{R}^+;L^2(\Omega))} + \ep k_0\|\Delta w\|^2  \notag \\ 
& \le C_\ep(R) + \ep k_0\|\Delta w\|^2.  \label{qww-3.55}
\end{align}
Applying \eqref{qww-3.51}-\eqref{qww-3.55} to \eqref{qww-3.2} gives us the lower bound for all $\ep>0$,
\begin{align}
\Phi \ge (1-\ep)\|\nabla w_t\|^2 + (\ell_0-(2+k_0)\ep) \|\Delta w\|^2 + \|\zeta^t\|^2_{\mathcal{M}^1} - C_\ep(R).
\end{align}
For any fixed $0<\ep<\min\{1,\ell_0/(2+k_0)\}$, we obtain \eqref{qww-3.5}.

Returning to the aim of \eqref{qww-3}, we first add 
\[
3\|\nabla w_t\|^2 + 2(\nabla w_t,\nabla w)
\]
to both sides of \eqref{qww-2}, and also insert
\begin{align}
\int_0^\infty g(s) \frac{d}{ds}\|\zeta^t(s)\|^2_{\mathcal{V}^2_a} ds & = -\int_0^\infty g'(s) \|\zeta^t(s)\|^2_{\mathcal{V}^2_a} ds  \notag \\ 
& \ge \delta\int_0^\infty g(s) \|\zeta^t(s)\|^2_{\mathcal{V}^2_a} ds  \notag \\ 
& = \delta\|\zeta^t\|^2_{\mathcal{M}^1}.  \notag
\end{align}
Putting these together and using the second inequality in \eqref{kernel}, \eqref{qww-2} becomes the differential inequality
\begin{align}
& \frac{d}{dt}\Phi + \|\nabla w_t\|^2 + 2(\nabla w_t,\nabla w) + 2\|\Delta w\|^2 - 2k_0\|w\|^2_{\mathcal{V}^2_a} + \delta\|\zeta^t\|^2_{\mathcal{M}^1}  \notag \\
& + 2(b(x)w_t,(-\Delta)w) + 2k_0(\nabla a(x)\cdot\nabla w,(-\Delta)w) - 2\int_0^\infty g(s)(\nabla a(x)\cdot\nabla\zeta^t(s),(-\Delta)w(t)) ds  \notag \\
& \le 3\|\nabla w_t\|^2 + 2(\nabla w_t,\nabla w) + 2(b(x) w_{tt}, (-\Delta)w) + 2k_0(\nabla a(x)\cdot\nabla w_t,(-\Delta)w)  \notag \\
& - 2\int_0^\infty g(s)(a(x)(-\Delta)\zeta^t(s),(-\Delta)w(t)) ds - 2\int_0^\infty g(s)(\nabla a(x)\cdot \nabla \zeta^t_t(s),(-\Delta)w(t)) ds  \notag \\ 
& - 2(\psi'(w)\nabla w,\nabla w_t) - 2(\psi(w),(-\Delta)w) + 2\beta(\nabla u,\nabla w_t) + 2\beta(u,(-\Delta)w).  \label{qww-4}
\end{align}
(We should mention that the final bound of \eqref{lpp-12} is now realized to control the $\nabla\zeta^t_t$ term appearing on the right-hand side.)
Next we employ some basic inequalities, the assumptions on $a$ and $b$, the assumptions \eqref{assm-f-1}\eqref{cons-f-1}, the bounds \eqref{unif-bnd}, \eqref{wer-1} and \eqref{lpp-12}, and finally even the continuous embedding $\mathcal{V}^2_a\hookrightarrow H^1_0(\Omega)$ of (H1r) to control the right-hand side of \eqref{qww-4} with the estimates
\begin{align}
3\|\nabla w_t\|^2 + 2(\nabla w_t,\nabla w) - 2(\psi'(w)\nabla w,\nabla w_t) + 2\beta(\nabla u,\nabla w_t) \le C(R),  \label{qww-5}
\end{align}
\begin{align}
2(b(x) w_{tt},(-\Delta)w) & \le C(R) + \frac{1}{4}\|\Delta w\|^2,  \label{qww-6}
\end{align}
\begin{align}
2k_0(\nabla a(x)\cdot\nabla w_t,(-\Delta)w) & \le C(R) + \frac{1}{4}\|\Delta w\|^2,  \label{qww-7}
\end{align}
\begin{align}
-2\int_0^\infty g(s)(a(x)(-\Delta)\zeta^t(s),(-\Delta)w(t)) ds & = -2\int_0^\infty g(s)(\zeta^t(s),w(t))_{\mathcal{V}^2_a} ds  \notag \\ 
& \le 2\int_0^\infty g(s)\|\zeta^t(s)\|_{\mathcal{V}^2_a}\|w(t)\|_{\mathcal{V}^2_a} ds  \notag \\
& \le \ep \|\zeta^t\|^2_{\mathcal{M}^1} + \frac{2}{\ep k_0}k_0\|w\|^2_{\mathcal{V}^2_a},  \label{qww-8}  
\end{align}
\begin{align}
-2\int_0^\infty g(s)(\nabla a(x)\cdot \nabla \zeta^t_t(s),(-\Delta)w(t)) ds & \le 2\int_0^\infty g(s)\|\nabla a(x)\cdot \nabla \zeta^t_t(s)\| \|\Delta w(t)\| ds  \notag \\
& \le 2\int_0^\infty g(s)\|\nabla a\|_\infty \|\nabla \zeta^t_t(s)\| \|\Delta w(t)\| ds  \notag \\
& \le \frac{1}{\ep}\|\nabla a\|^2_\infty \int_0^\infty g(s) \|\nabla \zeta^t_t(s)\|^2 ds + \ep\int_0^\infty g(s) \|\Delta w(t)\|^2 ds  \notag \\
& = \frac{1}{\ep}\|\nabla a\|^2_\infty \|\zeta^t_t\|^2_{L^2_g(\mathbb{R}^+;H^1_0(\Omega))} + \ep k_0\|\Delta w\|^2  \notag \\
& \le C_\ep(R) \|\zeta^t_t\|^2_{\mathcal{M}^0} + \ep k_0\|\Delta w\|^2  \notag \\
& \le C_\ep(R) + \ep k_0\|\Delta w\|^2,  \label{qww-9}  
\end{align}
\begin{align}
-2(\psi(w),(-\Delta)w) & \le C(R) + \frac{1}{4}\|\Delta w\|^2,  \label{qww-10}
\end{align}
and
\begin{align}
2\beta(u,(-\Delta)w) & \le C(R) + \frac{1}{4}\|\Delta w\|^2.  \label{qww-11}
\end{align}
Hence, \eqref{qww-5}-\eqref{qww-11} show the right-hand side of \eqref{qww-4} is controlled with, for all $\ep>0,$
\[
C_\ep(R) + (1+\ep k_0)\|\Delta w\|^2 + \frac{2}{\ep k_0}k_0\|w\|^2_{\mathcal{V}^2_a} + \ep\|\zeta^t\|^2_{\mathcal{M}^1}.
\]
Now fixing $0<\ep<\min\{1/k_0,\delta\}$ and setting 
\[
m_2=m_2(k_0,\delta):=\min\{1-\ep k_0,\delta-\ep\}>0 \quad \text{and} \quad c=c(k_0):=2\left( 1+\frac{1}{\ep k_0} \right)
\]
we arrive at the desired estimate \eqref{qww-3}.

So now we integrate the linear differential inequality \eqref{qww-3} and apply $\Phi(0)=0$.
Thus,
\begin{align}
& \|\Delta w(t)\|^2 + \|\nabla w_t(t)\|^2 + \|\zeta^t\|^2_{\mathcal{M}^1} \le Q_{\delta}(R),  \label{qww-12}
\end{align}
for some positive nondecreasing function $Q_{\delta}(\cdot) \sim \delta^{-1}.$
By combining \eqref{qww-12}, \eqref{lpp-12} and the Poincar\'{e} inequality \eqref{Poincare}, we see that, with the $H^2$-elliptic regularity estimate \eqref{reg-est}, we have with uniform bounds 
\[
w(t)\in H^2(\Omega) \quad \text{and} \quad w_t(t)\in H^1(\Omega) \quad \forall\ t>0.
\]
Additionally, collecting the bounds \eqref{qww-12} and \eqref{part-z} establishes that, for all $t\ge0$, 
\begin{align}
\|\zeta^t\|^2_{\mathcal{M}^1} + \|\zeta^t_s\|^2_{\mathcal{M}^0} \le Q_{\delta}(R). \label{qww-99}
\end{align}
Lastly, to show \eqref{compact} holds we need to control the last term of the norm \eqref{new-norm}.
With the bound \eqref{unif-bnd}, we apply the conclusion of Lemma \ref{what-2} here in the form 
\begin{align}
\sup_{\tau\ge1} \tau\mathbb{T}(\tau;\zeta^t) & \le 2 \left( t+2 \right)e^{-\delta t} \sup_{\tau\ge1} \tau\mathbb{T}(\tau;\zeta_0) + C(R) \notag \\
& \le C(R).  \label{strong-bound-1} 
\end{align}
where the last inequality follows from the null initial condition given in \eqref{problem-w}$_4$.
Together, the estimates \eqref{qww-12}-\eqref{strong-bound-1} show that \eqref{compact} holds.
This completes the proof.
\end{proof}

We now prove the main theorem.

\begin{proof}[Proof of Theorem \ref{t:main-reg}]
Define the subset $\mathcal{C}$ of $\mathcal{K}^1$ by 
\begin{equation*}
\mathcal{C} := \{ U=(u,v,\eta) \in \mathcal{K}^1: \|U\|_{\mathcal{K}^1} \le Q(R) \},
\end{equation*}
where $Q(R)>0$ is the function from Lemma \ref{t:uniform-compactness}, and $R>0$ is such that $\|U_{0}\|_{\mathcal{H}^0} \le R.$ 
Let now $U_0=(u_{0},u_{1},\eta_0) \in \mathcal{B}$ (the bounded absorbing set of Corollary \ref{t:abs-set} endowed with the topology of $\mathcal{H}^0$). 
Then, for all $t\ge0$ and for all $U_{0}\in \mathcal{B}$, $S(t)U_{0} = Z(t)U_{0} + K(t)U_{0}$, where $Z(t)$ is uniformly and exponentially decaying to zero by Lemma \ref{t:uniform-decay}, and, by Lemma \ref{t:uniform-compactness}, $K(t)$ is uniformly bounded in $\mathcal{K}^1.$ 
In particular, there holds
\begin{equation*}
\mathrm{dist}_{\mathcal{H}^0}(S(t)\mathcal{B},\mathcal{C}) \le Q(R)e^{-\omega t}.
\end{equation*}
The proof is finished.
\end{proof}

\section{Conclusions}

We have show that the global attractors associated with a wave equation with degenerate viscoelastic dissipation in the form of degenerate memory possesses more regularity than previously obtained in \cite{CFM-16}.
This is established under reasonable assumptions by showing the existence of a compact attracting set to which global attractor resides.
Moreover, the global attractor consists of regular solutions.
The main difficulties encountered here are due to the degeneracy of the dissipation term as well as obtaining compactness for the memory term.

\appendix  \label{appendix}
\section{}

We include two frequently used Gr\"{o}nwall-type inequalities that are important to this paper.
The first can be found in \cite[Lemma 5]{Pata-Zelik-06}; the second in \cite[Lemma 2.2]{Grasselli&Pata04}.

\begin{proposition}  \label{GL}
Let $\Lambda :\mathbb{R}^+\rightarrow \mathbb{R}^+$ be an absolutely continuous function satisfying 
\begin{equation*}
\frac{d}{dt}\Lambda(t) + 2\eta \Lambda(t) \le h(t)\Lambda(t)+k,
\end{equation*}
where $\eta>0$, $k\ge0$ and $\int_s^t h(\tau)d\tau \le \eta(t-s)+m$, for all $t\ge s\ge0$ and some $m\ge 0$. 
Then, for all $t\ge0$, 
\begin{equation*}
\Lambda(t) \le \Lambda(0)e^{m}e^{-\eta t}+\frac{ke^m}{\eta}.
\end{equation*}
\end{proposition}

\begin{proposition}  \label{GL2}
Let $\Phi:[0,\infty)\rightarrow[0,\infty)$ be an absolutely continuous function such that, for some $\ep>0$,
\[
\frac{d}{dt}\Phi(t)+2\ep\Phi(t)\le f(t)\Phi(t)+h(t)
\]
for almost every $t\in[0,\infty)$, where $f$ and $h$ are functions on $[0,\infty)$ such that 
\[
\int_s^t|f(\tau)|d\tau \le \alpha(1+(t-s)^\lambda), \quad \sup_{t\ge0}\int_t^{t+1}|h(\tau)|d\tau\le\beta
\]
for some $\alpha,\beta\ge0$ and $\lambda\in[0,1)$.
Then
\[
\Phi(t)\le\gamma\Phi(0)e^{-\ep t}+K
\]
for every $t\in[0,\infty)$, for some $\gamma=\gamma(f,\ep,\lambda)\ge1$ and $K=K(\ep,\lambda,f,h)\ge0.$
\end{proposition}

\section*{Acknowledgments}

The author is indebted to the anonymous referees for their careful reading of the manuscript and for their helpful comments and suggestions---in particular, for the reference \cite{CRV-08}.

\bigskip

\bibliographystyle{amsplain}
\providecommand{\bysame}{\leavevmode\hbox to3em{\hrulefill}\thinspace}
\providecommand{\MR}{\relax\ifhmode\unskip\space\fi MR }
\providecommand{\MRhref}[2]{%
  \href{http://www.ams.org/mathscinet-getitem?mr=#1}{#2}
}
\providecommand{\href}[2]{#2}

\end{document}